\documentclass[bibalpha]{amsart}
\usepackage{amssymb}
\usepackage{textcomp}
\usepackage[mathscr]{euscript}
\usepackage{pb-diagram}

\newtheorem{theorem}{Theorem}[section]

\newtheorem{prop}[theorem]{Proposition}
\newtheorem{lem}[theorem]{Lemma}

\newtheorem{fact}[theorem]{Fact}
\newtheorem{cor}[theorem]{Corollary}

\newtheorem*{conj}{Conjecture}
\newtheorem*{thmA}{Theorem A}

\newtheorem*{corA}{Corollary A}

\theoremstyle{definition}

\theoremstyle{remark}

\ProvideTextCommandDefault{\cprime}{(U+042C)}

\newcommand{\Cal}[1]{\ensuremath{\mathcal{#1}}}

\newcommand{\Z}{\mathbb{Z}}

\newcommand{\C}{\mathbb{C}}
\newcommand{\Q}{\mathbb{Q}}
\newcommand{\R}{\mathbb{R}}

\DeclareMathOperator{\Gl}{Gl}

\begin{document}
\title{Expansions of the real field by discrete subgroups of $\Gl_n(\mathbb{C})$}


\thanks{This is a preprint version. Later versions might contain significant changes. The first author was partially supported by NSF grant DMS-1654725.
The second author was partially supported by the European Research Council under
the European Union’s Seventh Framework Programme (FP7/2007-2013) / ERC Grant agreement
no. 291111/ MODAG }

\author{Philipp Hieronymi}
\address
{Department of Mathematics\\University of Illinois at Urbana-Champaign\\1409 West Green Street\\Urbana, IL 61801}
\email{phierony@illinois.edu}
\urladdr{http://www.math.illinois.edu/\textasciitilde phierony}

\author{Erik Walsberg}
\address
{Department of Mathematics\\University of Illinois at Urbana-Champaign\\1409 West Green Street\\Urbana, IL 61801}
\email{erikw@illinois.edu}
\urladdr{http://www.math.illinois.edu/\textasciitilde erikw}

\author{Samantha Xu}
\address
{School of Social Work\\University of Illinois at Urbana-Champaign\\1010 West Nevada Street\\Urbana, IL 61801}
\email{samxu@illinois.edu}

\date{\today}

\maketitle

\begin{abstract}
Let $\Gamma$ be an infinite discrete subgroup of Gl$_n(\mathbb{C})$.
Then either $(\mathbb{R},<,+,\cdot,\Gamma)$ is interdefinable with $(\mathbb{R},<,+,\cdot, \lambda^{\mathbb{Z}})$ for some $\lambda \in \mathbb{R}$, or $(\mathbb{R},<,+,\cdot,\Gamma)$ defines the set of integers.
When $\Gamma$ is not virtually abelian, the second case holds.
\end{abstract}

\section{Introduction}
\noindent Let $\bar{\mathbb{R}} = (\mathbb{R},<,+,\cdot,0,1)$ be the real field. For $\lambda \in \mathbb{R}_{>0}$, set $\lambda^{\mathbb{Z}} := \{ \lambda^m : m \in \mathbb{Z}\}$.
Throughout this paper $\Gamma$ denotes a discrete subgroup of Gl$_{n}(\mathbb{C})$, and $G$ denotes a subgroup of Gl$_n(\mathbb{C})$.
We identify the set M$_n(\mathbb{C})$ of $n$-by-$n$ complex matrices with $\mathbb{C}^{n^2}$ and identify $\mathbb{C}$ with $\mathbb{R}^2$ in the usual way. Our main result is the following classfication of expansions of $\bar{\R}$ by a discrete subgroup of $\Gl_n(\C)$. 

\begin{thmA}\label{thm:main}
Let $\Gamma$ be an infinite discrete subgroup of $\Gl_n(\C)$. Then either
\begin{itemize}
\item $(\R,\Gamma)$ defines $\Z$ or
\item there is $\lambda \in \R_{>0}$ such that $(\bar{\R},\Gamma)$ is interdefinable with $(\bar{\R},\lambda^{\Z})$.
\end{itemize}
If $\Gamma$ is not virtually abelian, then $(\bar{\mathbb{R}},\Gamma)$ defines $\Z$.
\end{thmA}

\noindent By Hieronymi~\cite[Theorem 1.3]{discrete}, the structure $(\bar{\R},\lambda^{\Z},\mu^{\Z})$ defines $\Z$ whenever $\log_{\lambda} \mu \notin \mathbb{Q}$, and is interdefinable with $(\bar{\R},\lambda^{\Z})$ otherwise. Therefore Theorem A extends immediately to expansions of $\bar{\R}$ by multiple discrete subgroups of $\Gl_n(\C)$.

\begin{corA} Let $\Cal G$ be a collection of infinite discrete subgroups of various $\Gl_n(\C)$. Then either
\begin{itemize}
\item $(\bar{\mathbb{R}},\big(\Gamma\big)_{\Gamma \in \Cal G})$ defines $\Z$ or
\item there is $\lambda \in \R_{>0}$ such that $(\bar{\mathbb{R}},\big(\Gamma\big)_{\Gamma \in \Cal G})$ is interdefinable with $(\bar{\R},\lambda^{\Z})$.
\end{itemize}
\end{corA}

\noindent The dichotomies in Theorem A and Corollary A are arguably as strong as they can be. An expansion of the real field that defines $\Z$, has not only an undecidable theory, but also defines every real projective set in sense of descriptive set theory (see Kechris \cite[37.6]{kechris}). From a model-theoretic/geometric point of view such a structure is a wild as can be. On the other hand, by van den Dries \cite{vdd-Powers2} the structure $(\bar{\R},\lambda^{\Z})$ has a decidable theory whenever $\lambda$ is recursive, and admits quantifier-elimination in a suitably extended language. It satisfies combinatorical model-theoretic tameness conditions such as NIP and distality (see \cite{GH-Dependent, HN-distal}). Furthermore, it follows from these results that every subset of $\R^n$ definable in $(\bar{\R},\lambda^{\Z})$ is a boolean combination of open sets, and thus $(\bar{\R},\lambda^{\Z})$ defines only sets on the lowest level of the Borel hierarchy. See Miller \cite{Miller-tame} for more on tameness in expansions of the real field.
\newline

\noindent Our proof of Theorem A relies crucially on the following two criteria for the definability of $\Z$ in expansions of the real field.

\begin{fact}\label{thm:hier}
Suppose $D \subseteq \mathbb{R}^k$ is discrete.
\begin{enumerate}
\item If $(\bar{\mathbb{R}},D)$ defines a subset of $\mathbb{R}$ that is dense and co-dense in a nonempty open interval, then $(\bar{\mathbb{R}},D)$ defines $\Z$.
\item If $D$ has positive Assouad dimension, then $(\bar{\mathbb{R}},D)$ defines $\Z$.
\end{enumerate}
\end{fact}

\noindent The first statement is \cite[Theorem E]{discrete2}, a fundamental theorem on first-order expansions of $\bar{\mathbb{R}}$, and the second claim is proven using the first in Hieronymi and Miller \cite[Theorem A]{HM}.
We recall the definition of Assouad dimension in Section \ref{section:exp}.
This important metric dimension bounds more familiar metric dimensions (such as Hausdorff and Minkowski dimension) from above.
We refer to \cite{HM} for a more detailed discussion of Assouad dimension and its relevance to definability theory.\newline

\noindent The outline of our proof of Theorem A is as follows. Let $\Gamma$ be a discrete, infinite subgroup of $\Gl_n(\C)$. Using Fact \ref{thm:hier}(1), we first show that $(\bar{\R},\Gamma)$ defines $\Z$ whenever $\Gamma$ contains a non-diagonalizable matrix. It follows from a theorem of Mal\textquotesingle tsev that $(\bar{\R},\Gamma)$ defines $\Z$ when $\Gamma$ is virtually solvable and not virtually abelian. In the case that $\Gamma$ is not virtually solvable, we prove using Tits' alternative that $\Gamma$ has positive Assouad dimension, and hence $(\bar{\R},\Gamma)$ defines $\Z$ by Fact \ref{thm:hier}(2). We conclude the proof of Theorem A by proving that whenever $\Gamma$ is virtually abelian and $(\bar{\R},\Gamma)$ does not define $\Z$, then $(\bar{\R},\Gamma)$ is interdefinable with $(\bar{\R},\lambda^\Z)$ for some $\lambda \in \R_{>0}$. 
Along the way we give (Lemma~\ref{prop:heisen0}) an elementary proof showing that a torsion free non abelian nilpotent subgroup of Gl$_n(\mathbb{C})$ has a non-diagonalizable element.
As every finitely generated subgroup of Gl$_n(\mathbb{C})$ is either virtually nilpotent or has exponential growth, this yields a more direct proof of Theorem A in the case when $\Gamma$ is finitely generated. \newline

\noindent We want to make an extra comment about the case when $\Gamma$ is a discrete, virtually solvable, and not virtually abelian subgroup of $\Gl_n(\C)$. The \textbf{Novosibirsk theorem} \cite{Noskov} of Noskov (following work of Mal\textquotesingle stev, Ershov, and Romanovskii) shows that a finitely generated, virtually solvable and non-virtually abelian group interprets $(\mathbb{Z},+,\cdot)$.
It trivially follows that if $G$ is finitely generated, virtually solvable, and non-virtually abelian, then $(\bar{\mathbb{R}},G)$ interprets $(\mathbb{Z},+,\cdot)$.
However, it does not directly follow that $(\bar{\mathbb{R}},G)$ defines $\Z$. We use an entirely different method below to show that if $G$ is in addition discrete, then $(\bar{\mathbb{R}},G)$ defines $\Z$. Our method also applies when $G$ is not finitely generated, but relies crucially on the discreteness of $G$.\newline

\noindent This paper is by no means the first paper to study expansions of the real field by subgroups of $\Gl_n(\C)$. Indeed, there is a large body of work on this subject, often not explicitly mentioning $\Gl_n(\C)$. Because we see this paper as part of a larger investigation, we survey some of the earlier results and state a conjecture. It is convenient to consider three disctinct classes of such expansion. By Miller and Speissegger~\cite{MS99} every first-order expansion $\mathscr{R}$ of $\bar{\mathbb{R}}$ satisfies at least one of the following: 
\begin{enumerate}
\item $\mathscr{R}$ is o-minimal,
\item $\mathscr{R}$ defines an infinite discrete subset of $\mathbb{R}$,
\item $\mathscr{R}$ defines a dense and co-dense subset of $\mathbb{R}$.
\end{enumerate}
The \textbf{open core} $\mathscr{R}^\circ$ of $\mathscr{R}$ is the expansion of $(\R,<)$ generated by all open $\mathscr{R}$-definable subsets of all $\mathbb{R}^k$. By \cite{MS99}, if $\mathscr{R}$ does not satisfy (2), then $\mathscr{R}^\circ$ is o-minimal.\newline

The case when $\mathscr{R}$ is o-minimal, is largely understood.
Wilkie's famous theorem \cite{wilkie} that $(\bar{\mathbb{R}},\exp)$ is o-minimal is crucial.
This shows the expansion of $\bar{\mathbb{R}}$ by the subgroup
$$ \left\{ \begin{pmatrix}
1 & 0 & t\\[3pt]
0 & \lambda^t & 0\\[3pt]
0 & 0 & 1
\end{pmatrix} : t \in \mathbb{R} \right\} $$
 is o-minimal for $\lambda \in \mathbb{R}_{>0}$, and so is the expansion of $\bar{\mathbb{R}}$ by any subgroup of the form
$$ \left\{ \begin{pmatrix} t^s & 0 \\[3pt] 0 & t^r \end{pmatrix} : t \in \mathbb{R}_{>0} \right\} $$
for $s,r \in \mathbb{R}_{>0}$. Indeed, by Peterzil, Pillary, and Starchenko~\cite{PPS-linear}, whenever an expansion $(\bar{\R},G)$ by a subgroup $G$ of $\Gl_n(\R)$ is o-minimal, then $G$ is already definable in $(\bar{\mathbb{R}},\exp)$. Futhermore, note that by a classical theorem of Tannaka and Chevalley~\cite{Chevalley} every compact subgroup of Gl$_n(\mathbb{C})$ is the group of real points on an algebraic group defined over $\mathbb{R}$. Thus every compact subgroup of $\Gl_n(\mathbb{C})$ is $\bar{\mathbb{R}}$-definable, and therefore the case of expansions by compact subgroups of $\Gl_n(\C)$ is understood as well.\newline

We now consider the case when infinite discrete sets are definable.  Corollary A for discrete subgroups of $\C^{\times}$ follows easily from the proof of \cite[Theorem 1.6]{discrete}. While Corollary A handles the case of expansions by discrete subgroups of $\Gl_n(\C)$, there are examples of subgroups of $\Gl_n(\C)$ that define infinite discrete sets, but fail the conclusion of Theorem A. Given $\alpha \in \mathbb{R}^\times$ the logarithmic spiral
$$
 S_\alpha  = \{ (\exp( t) \sin(\alpha t), \exp(t) \cos(\alpha t) ) : t \in \mathbb{R} \}
$$
is a subgroup of $\mathbb{C}^\times$.
Let $\mathfrak{s}$ and $\mathfrak{e}$ be the restrictions of $\sin$ and $\exp$ to $[0,2\pi]$, respectively.
Then $(\bar{\mathbb{R}},S_\alpha)$ is a reduct of $( \bar{\mathbb{R}}, \mathfrak{s}, \mathfrak{e}, \lambda^{\mathbb{Z}})$ when $\lambda = \exp(2\pi\alpha)$, as was first observed by Miller and Speissegger.
As $(\bar{\mathbb{R}},\mathfrak{s},\mathfrak{u})$ is o-minimal with field of exponents $\mathbb{Q}$, the structure $(\bar{\mathbb{R}},S_\alpha)$ is d-minimal\footnote{A expansion $\mathscr{R}$ of $\bar{\mathbb{R}}$ is \textbf{d-minimal} if every definable unary set in every model of the theory of $\mathscr{R}$ is a union of an open set and finitely many discrete sets.} by Miller \cite[Theorem 3.4.2]{Miller-tame} and thus does not define $\Z$. It can be checked that $(\bar{\mathbb{R}},S_\alpha)$ defines a analytic function that is not semi-algebraic\footnote{By induction on the complexity of terms it follows easily from [Theorem II, vdD] that the definable functions in $(\bar{\R},\lambda^{\Z})$ are given piecewise by a finite compositions of $x\mapsto \max \Big(\{0\} \cup \big(\lambda^{\Z} \cap [-\infty,x]\big)\Big)$ and functions definable in $\bar{\R}$. From this one can deduce that every definable function in this structure is piecewise semi-algebraic.}, and thus is not interdefinable with $(\bar{\R},\lambda^{\Z})$ for any $\lambda \in \R_{>0}$. \newline

Most work in the case of expansions that define dense and co-dense sets, concerns expansions by finite rank subgroups of $\mathbb{C}^\times$ (see introduction of \cite{Erin-thesis} for a thorough discussion of expansions by subgroups of $\mathbb{C}^\times$).
In \cite{vdDG-groups} van den Dries and G\"{u}nayd\i n showed that an expansion of $\bar{\mathbb{R}}$ by a finitely generated dense subgroup of $(\mathbb{R}_{>0},\cdot)$ admits quantifier-elimination in a suitably extend language. G\"{u}nayd\i n~\cite{Ayhan-thesis} and Belegradek and Zilber~\cite{Belegradek-Zilber} proved similar results for the expansion of $\bar{\mathbb{R}}$ by a dense finite rank subgroup of the unit circle $\mathbb{U} := \{ a \in \mathbb{C}^\times : |a| = 1\}$. This covers the case when $G$ is the group of roots of unity. In all these cases the open core of the resulting expansion is interdefinable with $\bar{\mathbb{R}}$. This does not always have to be the case. In Caulfield~\cite{Erin-paper} studies expansions by subgroups of $\mathbb{C}^\times$ of the form
$$\{ \lambda^{k} \exp(i\alpha l) : k,l \in \mathbb{Z}\} \quad \text{where  } \lambda \in \mathbb{R}_{>0}  \text{  and  } \alpha \in \mathbb{R} \setminus \pi \mathbb{Q}.$$
Such an expansions obviously defines a dense and co-dense subset of $\R$, but by \cite{Erin-paper} its open core is interdefinable with $(\bar{\mathbb{R}},\lambda^\Z)$. Futhermore, even if the open core is o-minimal, it does not have to be interdefinable with $\bar{\R}$. By \cite{H-tau} there is a co-countable subset $\Lambda$ of $\mathbb{R}_{>0}$ such that if $r \in \Lambda$ and $H$ is a finitely generated dense subgroup of $(\mathbb{R}_{>0},\cdot)$ contained in the algebraic closure of $\mathbb{Q}(r)$, then the open core of the expansion of $\bar{\mathbb{R}}$ by the subgroup
$$ \left\{ \begin{pmatrix} t & 0 \\[3pt] 0 & t^r \end{pmatrix} : t \in H \right\}  $$
 is interdefinable with the expansion of $\bar{\mathbb{R}}$ by the power function $t \mapsto t^r : \mathbb{R}_{>0} \to \mathbb{R}_{>0}$.\newline 

All these previous results suggest that the next class of subgroups of $\Gl_n(\C)$ for which we can hope to prove a classification comparable to Theorem A, is the class of finitely generated subgroups. Here the following conjecture seems natural, but most likely very hard to prove. Let $\bar{\mathbb{R}}_{\text{Pow}}$ be the expansion of $\bar{\mathbb{R}}$ by all power functions $\mathbb{R}_{>0} \to \mathbb{R}_{>0}$ of the form $t \mapsto t^r$ for $r \in \mathbb{R}^\times$.

\begin{conj}
Let $G$ be a finitely generated subgroup of $\Gl_n(\C)$ such that $(\bar{\mathbb{R}},G)$ does not define $\Z$. Then the open core of $(\bar{\mathbb{R}},G)$ is a reduct of $\bar{\mathbb{R}}_{\textrm{Pow}}$ or of $(\bar{\mathbb{R}},S_\alpha)$ for some  $\alpha \in \mathbb{R}_{>0}$.
\end{conj}

Even when the statement ``$(\bar{\mathbb{R}},G)$ does not define $\Z$'' is replaced by  ``$(\bar{\mathbb{R}},G)$ does not interpret $(\mathbb{Z},+,\cdot)$'', the conjecture is open. However, this weaker conjecture might be easier to prove, because the Novosibirsk theorem can be used to rule out the case when $G$ is virtually solvable and non-virtually abelian. It is worth pointing out that Caulfield conjectured that when $G$ is assumed to be a subgroup of $\C^{\times}$, then the open core $(\bar{\mathbb{R}},G)$ is either $\bar{\R}$ or a reduct of $(\bar{\mathbb{R}},S_\alpha)$ for some  $\alpha \in \mathbb{R}_{>0}$. See \cite{Erin-paper, Erin-thesis} for progress towards this later conjecture.

\section{Notation and Conventions}
\noindent Throughout $m,n$ range over $\mathbb{N}$ and $k,l$ range over $\mathbb{Z}$, $G$ is a subgroup of Gl$_n(\mathbb{C})$, and $\Gamma$ is a discrete subgroup of Gl$_n(\mathbb{C})$. Let $\bar{\mathbb{R}}_\Gamma$ be the expansion of $\bar{\mathbb{R}}$ by a $(2n)^2$-ary predicate defining $\Gamma$.
We set $\bar{\mathbb{R}}_\lambda := \bar{\mathbb{R}}_{\lambda^{\Z}}$.
A subset of $\mathbb{R}^k$ is \textbf{discrete} if every point is isolated.
We let UT$_n(\mathbb{C})$ be the group of $n$-by-$n$ upper triangular matrices, D$_n(\mathbb{C})$ be the group of $n$-by-$n$ diagonal matrices, and $\mathbb{U}$ be the multiplicative group of complex numbers with norm one.\newline

All structures considered are first-order, ``definable''  means ``definable, possibly with parameters''.
Two expansions of $(\mathbb{R},<)$ are \textbf{interdefinable} if they define the same subsets of $\mathbb{R}^k$ for all $k$.
If P is a propety of groups then a group $H$ is \textbf{virtually P} if there is finite index subgroup $H'$ of $H$ that is P.

\section{Linear Groups}
\noindent We gather some general facts on groups.
Throughout this section $H$ is a finitely generated group with a symmetric set $S$ of generators.
Let $S_m$ be the set of $m$-fold products of elements of $S$ for all $m$.
If $S'$ is another symmetric set of generators then there is a constant $k \geq 1$ such that
$$ k^{-1} |S_m| \leq |S'_m| \leq k |S_m| \quad \text{for all } m.$$
Thus the growth rate of $m \mapsto | S_m |$ is an invariant of $H$.
We say $H$ has \textbf{exponential growth} if there is a $C \geq 1$ such that $|S_m| \geq C^m$ for all $m$ and $H$ has \textbf{polynomial growth} there are $k, t \in \mathbb{R}_{>0}$ such that $|S_m| \leq t m^k$ for all $m$.
Note finitely generated non-abelian free groups are of exponential growth.
Gromov's theorem~\cite{Gromov-poly} says $H$ has polynomial growth if and only if it is virtually nilpotent.
Gromov's theorem for subgroups of Gl$_n(\mathbb{C})$ is less difficult and may be proven using the following two theorems:

\begin{fact}\label{thm:tits}
If $G$ does not contain a non-abelian free subgroup, then $G$ is virtually solvable.
\end{fact}

\noindent Fact~\ref{thm:tits} is \textbf{Tits' alternative}~\cite{tits}.
Fact~\ref{thm:bmw} is due to Milnor~\cite{Milnor-solv} and Wolf~\cite{Wolf-solv}.

\begin{fact}\label{thm:bmw}
Suppose $H$ is virtually solvable.
Then $H$ either has exponential or polynomial growth.
If the latter case holds then $H$ is virtually nilpotent.
\end{fact}

\noindent Note Fact \ref{thm:tits} and Fact \ref{thm:bmw} imply every finitely generated subgroup of Gl$_n(\mathbb{C})$ is of polynomial or exponential growth.
This dichotomy famously does not hold for finitely generated groups in general, see for example~\cite{Grigorchuk}.

\medskip \noindent The \textbf{Heisenberg group} $\mathbb{H}$ is presented by generators $a,b,c$ and relations
$$ [a,b] = c, \quad ac = ca, \quad bc = cb. $$
The following fact is folklore; we include a proof for the reader.

\begin{fact}\label{fact:heisen}
Let $E$ be a nilpotent, torsion-free, and non-abelian group.
Then there is a subgroup of $E$ isomorphic to $\mathbb{H}$.
\end{fact}

\begin{proof}
Let $e$ be the identity element of $E$.
We define the lower central series $(E_k)_{k \in \mathbb{N}}$ of $E$ by declaring $E_0 = E$ and $E_{k} = [E_{k-1},E]$ for $k \geq 1$.
Nilpotency means there is an $m$ such that $E_m \neq \{e\}$ and $[E_m,E] = \{e\}$.
Moreover $m \geq 1$ as $E$ is not abelian.

On one hand, $[E_{m-1},E] = E_m \neq \{e\}$ and so $E_{m-1}$ is not contained in $Z(E)$.
Thus, there exists $a \in E_{m-1} \setminus Z(E)$ and $b \in E_m$ that does not commute with $a$.
On the other hand, $[E_m,E]=\{e\}$ implies $E_m$ is contained in the center $Z(E)$ of $E$ and is thus abelian. So, $c := [a,b]$ is an element of $Z(E)$ and commutes with both $a$ and $b$.

Finally, $a,b,c$ have infinite order because $E$ is torsion-free. So, $a,b,c$ generate a subgroup of $E$ isomorphic to the Heisenberg group.
\end{proof}

\subsection{Non-diagonalizable elements}
We show certain linear groups necessarily contain non-diagonalizable elements.

\begin{lem}\label{prop:heisen0}
If $G$ is nilpotent, torsion-free, and not abelian, then $G$ contains a non-diagonalizable element.
\end{lem}

\noindent Lemma~\ref{prop:heisen0} follows from  Fact~\ref{fact:heisen} above and Lemma~\ref{lem:heisen1} below.

\begin{lem}\label{lem:heisen1}
Suppose $a,b,c \in \text{Gl}_n(\mathbb{C})$ satisfy
$$ [a,b] = c, \quad ac = ca, \quad bc  =cb, $$
and $c$ is not torsion.
Then either $a$ or $c$ is not diagonalizable.
\end{lem}

\begin{proof}
Suppose $a,c$ are both diagonalizable.
As $a,c$ commute, they are simultaneously diagonalizable and share a basis $\mathfrak{B}$ of eigenvectors.
As $c$ is not torsion, there is $\lambda_c \in \mathbb{C}^\times$ which is not a root of unity and $v \in \mathfrak{B}$ such that $cv = \lambda_c v$.
 Let $\lambda_a \in \mathbb{C}^\times$ be such that $a v = \lambda_a v$.

By way of contradiction, we will show $a (b^{k} v)  = (\lambda_a \lambda_c^k)(b^{k}v)$ for all $k \geq 1$.
As $\lambda_c$ is not a root of unity, this implies $a$ has infinitely many eigenvalues, which is impossible for an $n \times n$ matrix.
The base case holds as
$$ a(bv) = bac v = (\lambda_a \lambda_c) (bv).$$
Let $k \geq 2$ and suppose $a (b^{k-1} v) = (\lambda_a \lambda_c^{k-1}) (b^{k-1} v)$.
As $c$ commutes with $b$,
$$ a(b^{k} v) = ab( b^{k-1} v) = bac( b^{k-1} v) = bab^{k-1}cv = (\lambda_c) (bab^{k-1}v). $$
Applying the inductive assumption,
\begin{align*}
(\lambda_c) (bab^{k-1}v) &= (\lambda_c) b ( \lambda_a \lambda_c^{k-1} b^{k-1} v) = (\lambda_a \lambda_c^{k}) (b^{k} v).
 \end{align*}
\end{proof}

\noindent We now prove a slight weakening of Lemma~\ref{prop:heisen0} for solvable groups.
Recall $a \in \text{Gl}_n(\mathbb{C})$ is \textbf{unipotent} if some conjugate of $a$ is upper triangular with every diagonal entry equal to one.
The only diagonalizable unipotent matrix is the identity.
We recall a theorem of Mal\textquotesingle tsev \cite{Mal-sol}.

\begin{fact}\label{thm:lkm}
Suppose $G$ is solvable.
Then there is a finite index subgroup $G'$ of $G$ such that $G'$ is conjugate to a subgroup of UT$_n(\mathbb{C})$.
\end{fact}

\noindent We now derive an easy corollary from Fact \ref{thm:lkm}

\begin{lem}\label{prop:solv-1}
Suppose $G$ is solvable and not virtually abelian.
Then $G$ contains a non-diagonalizable element.
\end{lem}

\begin{proof}
Suppose every element of $G$ is diagonalizable.
After applying Fact \ref{thm:lkm} and making a change of basis if necessary we suppose $G' = G \cap \text{UT}_n(\mathbb{C})$ has finite index in $G$.
Let $\rho : \text{UT}_n(\mathbb{C}) \to \text{D}_n(\mathbb{C})$ be the natural quotient map; that is the restriction to the diagonal.
Every element of the kernel of $\rho$ is unipotent.
Thus the restriction of $\rho$ to $G'$ is injective, and so $G'$ is abelian.
\end{proof}

\section{Non-diagonalizable matrices}

\begin{lem}\label{lem:diag}
Suppose $G$ contains a non-diagonalizable matrix.
Then there is a rational function $h$ on $\text{Gl}_n(\mathbb{C}) \times \text{Gl}_{n}(\mathbb{C})$ such that $h(G \times G) \subseteq \mathbb{C}$ is dense in $\mathbb{R}_{>0}$.
\end{lem}

\begin{proof}
Suppose $a \in G$ is non-diagonalizable.
Let $b \in \text{Gl}_n(\mathbb{C})$ be such that $bab^{-1}$ is in Jordan form, i.e.
$$ bab^{-1} = \begin{pmatrix}  A_1 & O & \dots & O \\[4pt] O & A_2 & \ldots & O \\[4pt] \vdots &  \vdots & \ddots & \vdots \\[4pt]
O & O & \ldots & A_l \end{pmatrix} $$
where each $A_i$ is a Jordan block and each $O$ is a zero matrix of the appropriate dimensions.
We have
$$ b a^k b^{-1} = \begin{pmatrix}  A^k_1 & O & \dots & O \\[4pt] O & A^k_2 & \ldots & O \\[4pt] \vdots &  \vdots & \ddots & \vdots \\[4pt]
O & O & \ldots & A^k_l \end{pmatrix} \quad \text{for all } k. $$
As $a$ is not diagonalizable, $A_k$ has more then one entry for some $k$.
We suppose $A_1$ is $m$-by-$m$ with $m \geq 2$.
For some $\lambda \in \mathbb{C}^\times$ we have
$$ A_1 = \begin{pmatrix} \lambda & 1 & 0 & \ldots & 0 & 0 \\[4pt] 0 & \lambda & 1 & \ldots &  0 & 0 \\[4pt] 0 & 0 & \lambda & \ldots & 0 & 0 \\[4pt] \vdots & \vdots & \vdots & \ddots & \vdots & \vdots \\[4pt] 0 & 0 & 0 & \ldots & \lambda  & 1 \\ 0 & 0 & 0  & \ldots & 0 & \lambda \\
  \end{pmatrix}. $$
It is well-known and easy to show by induction that for every $k \geq 1$:
  $$ A^k_{1} = \begin{pmatrix}
\lambda^k & \binom{k}{1} \lambda^{k-1} & \binom{k}{2} \lambda^{k-2} & \binom{k}{3} \lambda^{k-3} & \binom{k}{4} \lambda^{k-4} & \ldots & \binom{k}{m} \lambda^{k-m+1} \\[4pt]
0 & \lambda^k & \binom{k}{1} \lambda^{k-1} & \binom{k}{2} \lambda^{k-2} &  \binom{k}{3} \lambda^{k-3} & \ldots & \binom{k}{m - 1} \lambda^{k-m +2}\\[4pt]
0 & 0 & \lambda^k & \binom{k}{1} \lambda^{k-1} & \binom{k}{2} \lambda^{k-2} & \ldots & \binom{k}{m - 2} \lambda^{k - m +3} \\[4pt]
0 & 0 & 0 & \lambda^{k} & \binom{k}{1} \lambda^{k-1} & \ldots & \binom{k}{m - 3} \lambda^{k - m +4} \\[4pt]
0 & 0 & 0 & 0 & \lambda^k & \ldots & \binom{k}{m - 4} \lambda^{k - m + 5} \\[4pt]
\vdots & \vdots & \vdots & \vdots & \vdots & \ddots & \vdots \\
0 & 0 & 0 & \vdots & \vdots & \lambda^k  & \binom{k}{1} \lambda^{k-1} \\[4pt]
0 & 0 & 0  & \ldots & \ldots & 0 & \lambda^k \\
  \end{pmatrix}. $$
  Let $g_{ij}$ be the $(i,j)$-entry of $g \in \text{Gl}_n(\mathbb{C})$.
Thus, for each $k \geq 1$,
$$(ba^{k}b^{-1})_{01} = k\lambda^{k-1} \quad \text{and} \quad (ba^{k}b^{-1})_{11} = \lambda^{k}.$$
We define a rational function $h'$ on $\Gl_n(\mathbb{C}) \times \Gl_n(\mathbb{C})$ by declaring
$$ h'(g,g') := \frac{ g_{01} g'_{11} }{ g'_{01} g_{11} } $$
for all $g,g' \in \Gl_n(\mathbb{C})$ such that $g_{11}, g'_{01} \neq 0$.
Then define $h$ by declaring
$$h(g,g') := h'(bgb^{-1}, bg'b^{-1})$$
We have
$$ h(a^i,a^j) = \frac{ (i \lambda^{i-1})(\lambda^j) }{ ( j \lambda^{j-1} ) (\lambda^i) } = \frac{i}{j} \quad \text{for all } i,j \geq 1.$$
Thus $\Q_{>0}$ is a subset of the image of $G \times G$ under $h$.
\end{proof}

\begin{cor}\label{cor:diag}
If $\Gamma$ contains a non-diagonalizable matrix, then $\bar{\mathbb{R}}_\Gamma$ defines $\Z$.
In particular, if $\Gamma$ is either
\begin{itemize}
\item solvable and not virtually abelian, or
\item torsion-free, nilpotent and non-abelian,
\end{itemize}
then $\bar{\mathbb{R}}_\Gamma$ defines $\Z$.
\end{cor}

\begin{proof}
Applying Lemma~\ref{lem:diag}, suppose $h$ is a rational function on Gl$_n(\mathbb{C}) \times \text{Gl}_n(\mathbb{C})$ such that the image of $\Gamma \times \Gamma$ under $h$ is dense in $\mathbb{R}_{>0}$.
Note $\Gamma$ is countable as $\Gamma$ is discrete.
It follows that the image of $\Gamma \times \Gamma$ under any function is co-dense in $\mathbb{R}_{>0}$. Fact~\ref{thm:hier}(1) implies that $\bar{\mathbb{R}}_\Gamma$ defines $\Z$. The second claim follows from the first by applying Lemma \ref{prop:heisen0} and Lemma \ref{prop:solv-1}.
\end{proof}

\begin{cor}\label{cor:cyclic}
If $a \in \text{Gl}_n(\mathbb{C})$ is non-diagonalizable, then $(\bar{\mathbb{R}}, \{ a^k : k \in \mathbb{Z}\})$ defines $\Z$.
\end{cor}
\begin{proof}
Set $G :=  \{ a^k : k \in \mathbb{Z}\}$.
The proof of Lemma~\ref{lem:diag} shows that in this case $\Q_{>0}$ is the intersection of $h(G \times G)$ and $\mathbb{R}_{>0}$. Thus the corollary follows by Julia Robinson's classical theorem of definability of $\Z$ in $(\mathbb{Q},+,\cdot)$ in \cite{julia}.
\end{proof}

\section{The case of exponential growth}\label{section:exp}
\noindent We recall the \textbf{Assouad dimension} of a metric space $(X,d)$.
See Heinonen \cite{Heinonen} for more information.
The Assouad dimension of a subset $Y$ of $\mathbb{R}^k$ is the Assouad dimension of $Y$ equipped with the euclidean metric induced from $\mathbb{R}^k$.

\medskip \noindent Suppose $A \subseteq X$ has at least two elements.
Then $A$ is $\delta$-\textbf{separated} for $\delta \in \mathbb{R}_{>0}$ if $d(a, b) \geq \delta$ for all distinct $a,b \in A$, and $A$ is \textbf{seperated} if $A$ is $\delta$-seperated for some $\delta > 0$.
Let $\mathscr{S}(A) \in \mathbb{R}$ be the supremum of all $\delta \geq 0$ for which $A$ is $\delta$-seperated.
Let $\mathscr{D}(A)$ be the \textbf{diameter} of $A$; that is the infimum of all $\delta \in \mathbb{R} \cup \{ \infty \}$ such that $d(a,b) < \delta$ for all $a,b \in A$, and $A$ is \textbf{bounded} if $\mathscr{D}(A) < \infty$.
Note $\mathscr{S}(A) \leq \mathscr{D}(A)$.
The \textbf{Assouad dimension} of $(X,d)$ is the infimum of the set of  $\beta \in \mathbb{R}_{>0}$ for which there is a $C > 0$ such that
$$ |A| \leq C \left( \frac{ \mathscr{D}(A) }{ \mathscr{S}(A) } \right)^\beta \text{   for all bounded and separated } A \subseteq X. $$
The proof of Fact~\ref{fact:use-assouad} is an elementary computation which we leave to the reader.

\begin{fact}\label{fact:use-assouad}
Suppose there is a sequence $\{A_m\}_{m \in \mathbb{N}}$ of bounded separated subsets of $X$ with cardinality at least two, and $B,C,t > 1$ are such that
$$ |A_m| \geq C^m \quad \text{and} \quad  \frac{ \mathscr{D}(A_m) }{ \mathscr{S}(A_m) } \leq t B^m \quad \text{for all } m $$
then $(X,d)$ has positive Assouad dimension.
\end{fact}

\medskip \noindent Let $|v|$ be the usual euclidean norm of $v \in \mathbb{C}^n$.
Given $g \in \text{M}_n(\mathbb{C})$ we let
$$ \| g \|  =\inf \{ t \in \mathbb{R}_{>0} : | gv | \leq t |v| \text{   for all } v \in \mathbb{C}^n \} $$
be the operator norm of $g$.
Then $\|\text{}\|$ is a linear norm on M$_{n}(\mathbb{C})$ and satisfies $\| g h \| \leq \| g \| \| h \|$ for all $g,h \in \text{M}_{n}(\mathbb{C})$. As any two linear norms on M$_{n}(\mathbb{C})$ are bi-Lipschitz equivalent the metric induced by $\| \|$ is bi-Lipschitz equivalent to the usual euclidean metric on $\mathbb{R}^{n^2}$.

\begin{prop}\label{prop:exp}
Suppose $\Gamma$ contains a finitely generated subgroup $\Gamma'$ of exponential growth.
Then $\Gamma$ has positive Assouad dimension.
\end{prop}

\begin{proof}
Because Assouad dimension is a bi-Lipschitz invariant (see \cite{Heinonen}), it suffices to show that $\Gamma$ has positive Assouad dimension with respect to the metric induced by $\| \|$. 
We let $I$ be the $n$-by-$n$ identity matrix.
Let $S$ be a symmetric generating set of $\Gamma'$, and let $S_m$ be the set of $m$-fold products of elements of $S$ for $m \geq 2$.
Set $$ B := \max\{ \| g \| : g \in S\} \quad \text{and} \quad D := \min \{ \| g - I \| : g \in \Gamma \}. $$
Note that $D > 0$, as $\Gamma$ is discrete, and that $B > 0$, as $\Gamma \neq \{I\}$.
Induction shows that $\| g \| \leq B^m$ when $g \in S_m$.
The triangle inequality directly yields $\mathscr{D}(S_m) \leq 2B^m$.
Each $S_m$ is symmetric as $S$ is symmetric.
Therefore $\| g^{-1} \| \leq B^m$ for all $g \in S_m$.
Let $g, h \in \Gamma$. We have
$$ \| I - g^{-1} h \| \leq \| g^{-1} \| \| g - h \|. $$
Equivalently,
$$ \frac{ \| I - g^{-1} h \| }{ \| g^{-1} \| } \leq \| g - h \|. $$

\medskip\noindent Suppose $g,h \in S_m$ are distinct.
Then $g^{-1}h \neq I$, and hence $\| I - g^{-1}h \| \geq D$.
So
$$ \| g - h \| \geq \frac{ \| I - g^{-1}h \| }{ \| g^{-1} \| } \geq \frac{ D }{ B^m }. $$
Hence $\mathscr{S}(S_m) \geq D/B^m$.
Thus
$$ \frac{ \mathscr{D}(S_m)}{ \mathscr{S}(S_m) } \leq \frac{ 2B^m }{ D/B^m} = \frac{ 2 }{D } B^{2m} .$$
As $\Gamma'$ has exponential growth, there is a $C > 0$ such that $|S_m| \geq C^m$ for all $m$.
An application of Fact~\ref{fact:use-assouad} shows that $\Gamma$ has positive Assouad dimension.
\end{proof}

\begin{prop}\label{prop:solv-2}
Suppose $\Gamma$ is not virtually abelian. Then $\bar{\mathbb{R}}_\Gamma$ defines $\Z$.
\end{prop}
\begin{proof}
By Corollary \ref{cor:diag}, we can assume that $\Gamma$ is solvable.
Thus by Fact \ref{thm:tits}, the group $\Gamma$ contains a non-abelian free subgroup. Therefore $\Gamma$ has positive Assouad dimension by Proposition~\ref{prop:exp}. We conclude that $\bar{\R}_{\Gamma}$ defines $\Z$ by Fact~\ref{thm:hier}(2). 
\end{proof}

\section{The virtually abelian case}
\noindent We first reduce the virtually abelian case to the abelian case.

\begin{lem}\label{lem:ab-1}
Suppose $G$ is  virtually abelian and every element of $G$ is diagonalizable.
Then there is a finite index abelian subgroup $G'$ of $G$ such that $(\bar{\mathbb{R}},G)$ and $(\bar{\mathbb{R}},G')$ are interdefinable.
\end{lem}

\begin{proof}
Let $G''$ be a finite index abelian subgroup of $G$.
As every element of $G''$ is diagonalizable, $G''$ is simultaneously diagonalizable.
Fix $g \in \text{Gl}_n(\mathbb{C})$ such that $gag^{-1}$ is diagonal for all $a \in G''$.
Let $G'$ be the set of $a \in G$ such that $gag^{-1}$ is diagonal, i.e. $G'$ is the intersection of $G$ and $g^{-1} \text{D}_n(\mathbb{C}) g$.
Then $G'$ is abelian, $(\bar{\mathbb{R}},G)$-definable, and is of finite index in $G$ as $G'' \subseteq G'$.
Because $G'$ has finite index in $G$, we have
$$ G = g_1 G' \cup \ldots \cup g_m G' \text{  for some  } g_1,\ldots,g_m \in G.$$
So $G$ is $(\bar{\mathbb{R}},G')$-definable.
\end{proof}

Proposition~\ref{prop:lastab} finishes the proof of Theorem A.

\begin{prop}\label{prop:lastab}
Suppose $\Gamma$ is abelian and $\bar{\mathbb{R}}_\Gamma$ does not define $\Z$.
Then there is $\lambda \in \R_{>0}$ such that $\bar{\mathbb{R}}_\Gamma$ is interdefinable with $\bar{\mathbb{R}}_\lambda$.
\end{prop}

\noindent Let $\mathfrak{u}: \mathbb{C}^\times \to \mathbb{U}$ be the argument map and $| \text{ } | : \mathbb{C}^\times \to \mathbb{R}_{>0}$ be the absolute value map. Thus $z = \mathfrak{u}(z)|z|$ for all $z \in \mathbb{C}^\times$.
Let $\mathbb{U}_m$ be the group of $m$th roots of unity for all $m \geq 1$.  In the following proof of Proposition \ref{prop:lastab} we will use the immediate corollary of \cite[Theorem 1.3]{discrete} that the structure $(\bar{\R},\lambda^{\Z},\mu^{\Z})$ defines $\Z$ whenever $\log_{\lambda} \mu \notin \mathbb{Q}$, and is is interdefinable with $(\bar{\R},\lambda^{\Z})$ otherwise.

\begin{proof}
Fact~\ref{thm:hier}(1) implies every countable $\bar{\mathbb{R}}_\Gamma$-definable subset of $\mathbb{R}$ is nowhere dense.
It follows that every $\bar{\mathbb{R}}_\Gamma$-definable countable subgroup of $\mathbb{U}$ is finite and every $\bar{\mathbb{R}}_\Gamma$-definable countable subgroup of $(\mathbb{R}_{>0},\cdot)$ is of the form $\lambda^{\mathbb{Z}}$ for some $\lambda \in \mathbb{R}_{>0}$.

Every element of $\Gamma$ is diagonalizable by Corollary~\ref{cor:diag}. Thus $\Gamma$ is simultaneously diagonalizable.
After making a change of basis we suppose $\Gamma$ is a subgroup of D$_n(\mathbb{C})$.
We identify D$_n(\mathbb{C})$ with $(\mathbb{C}^\times)^n$.
Let $\Gamma_i$ be the image of $\Gamma$ under the projection $(\mathbb{C}^\times)^n \to \mathbb{C}^\times$ onto the $i$th cordinate for $1 \leq i \leq n$.

Each $\mathfrak{u}(\Gamma_i)$ is finite.
Fix an $m$ such that $\mathfrak{u}(\Gamma_i)$ is a subgroup of $\mathbb{U}_m$ for all $1 \leq i \leq n$.
For each $1 \leq i \leq n$, $|\Gamma_i|$ is a discrete subgroup of $\mathbb{R}_{>0}$ and is thus equal to $\alpha_i^{\mathbb{Z}}$ for some $\alpha_i \in \mathbb{R}_{>0}$.
By \cite[Theorem 1.3]{discrete} each $\alpha_i$ is a rational power of $\alpha_1$.
Let $\lambda \in \mathbb{R}_{>0}$ be a rational power of $\alpha_1$ such that each $\alpha_i$ is an integer power of $\lambda$.
We show $\bar{\mathbb{R}}_\Gamma$ and $\bar{\mathbb{R}}_\lambda$ are interdefinable.
Note that $\lambda^{\mathbb{Z}}$ is $\bar{\mathbb{R}}_\Gamma$-definable; so it suffices to show $\Gamma$ is $\bar{\mathbb{R}}_\lambda$-definable.

Every element of $\Gamma_i$ is of the form $\sigma \lambda^k$ for some $\sigma \in \mathbb{U}_m$ and $k \in \mathbb{Z}$.
Thus $\Gamma$ is a subgroup of
$$ \Gamma' =  \left\{ \begin{pmatrix}  \sigma_1 \lambda^{k_1} & 0 & \dots & 0 \\[4pt] 0 & \sigma_2 \lambda^{k_2} & \ldots & 0 \\[4pt] \vdots &  \vdots & \ddots & \vdots \\[4pt]
0 & 0 & \ldots & \sigma_n \lambda^{k_n} \end{pmatrix} : \sigma_1, \ldots,\sigma_n \in \mathbb{U}_m, k_1,\ldots,k_n \in \mathbb{Z} \right\}. $$
Note $\Gamma'$ is $\bar{\mathbb{R}}_\lambda$-definable.
Abusing notation we let $\mathfrak{u} : (\mathbb{C}^\times)^n \to \mathbb{U}^n$ and we let $| \text{ } | : (\mathbb{C}^\times)^n \to (\mathbb{R}_{>0})^n$ be given by
$$ \mathfrak{u}(z_1,\ldots,z_n) = (\mathfrak{u}(z_1),\ldots,\mathfrak{u}(z_n)) \quad \text{and} \quad |(z_1,\ldots,z_n)| = (|z_1|,\ldots,|z_n|). $$
Then the map $(\mathbb{C}^\times)^n \to \mathbb{U}^{n} \times (\mathbb{R}_{>0})^n$ given by $\bar{z} \mapsto (\mathfrak{u}(\bar{z}), |\bar{z}|)$ restricts to a $\bar{\mathbb{R}}_\lambda$-definable isomorphism between $\Gamma'$ and $\mathbb{U}^n_m \times (\lambda^{\mathbb{Z}})^n$.
Lemma~\ref{lem:group} below implies any subgroup of $\mathbb{U}^n_m \times (\lambda^{\mathbb{Z}})^n$ is $\bar{\mathbb{R}}_\lambda$-definable.
\end{proof}

\noindent We consider $(\mathbb{Z}/m\mathbb{Z},+)$ to be a group with underlying set $\{0,\ldots,m-1\}$ in the usual way so that $(\mathbb{Z}/m\mathbb{Z},+)$ is a $(\mathbb{Z},+)$-definable group.
Lemma~\ref{lem:group} is folklore.
We include a proof for the sake of completeness.

\begin{lem}\label{lem:group}
Every subgroup $H$ of $(\mathbb{Z}/m\mathbb{Z})^l \times \mathbb{Z}^n$ for $l \geq 0$ is $(\mathbb{Z},+)$-definable.
\end{lem}

\begin{proof}
We first reduce to the case $l = 0$.
The quotient map $\mathbb{Z} \to \mathbb{Z}/m\mathbb{Z}$ is $(\mathbb{Z},+)$-definable, it follows that the coordinate-wise quotient $\mathbb{Z}^{l} \times \mathbb{Z}^{n} \to (\mathbb{Z}/m\mathbb{Z})^l \times \mathbb{Z}^n$ is $(\mathbb{Z},+)$-definable.
It suffices to show the preimage of $H$ in $\mathbb{Z}^{l + n}$ is $(\mathbb{Z},+)$-definable.

Suppose $H$ is a subgroup of $\mathbb{Z}^n$.
Then $H$ is finitely generated with generators $\beta_1,\ldots,\beta_k$ where
$ \beta_i = (b^{i}_1,\ldots,b^{i}_n) \text{  for all } 1 \leq i \leq k. $
Then
\begin{align*}
 H & = \left\{ \sum_{i = 1}^{k} c_i \beta_i  : c_1,\ldots,c_k \in \mathbb{Z}   \right\}  \\[4pt]
& = \left\{ \left( \sum_{i = 1}^{k} c_i b^{i}_{1} , \ldots , \sum_{i = 1}^{k} c_i b^{i}_n \right) : c_1,\ldots, c_n  \in \mathbb{Z}   \right\}.
\end{align*}
Thus $H$ is $(\mathbb{Z},+)$-definable.
\end{proof}

\section{countable $(\mathscr{R},\lambda^{\mathbb{Z}})$-definable groups}
\noindent Fix $\lambda \in \mathbb{R}_{>0}$ and an o-minimal $\mathscr{R}$ with  field of exponents $\Q$. Since $(\mathscr{R},\lambda^{\Z})$ does not define $\Z$ by \cite[Theorem 3.4.2]{Miller-tame}, Theorem~A implies every $(\mathscr{R},\lambda^{\mathbb{Z}})$-definable discrete subgroup of $\Gl_n(\mathbb{C})$ is virtually abelian. We extend this result to all countable interpretable groups.

\begin{prop}\label{prop:lam-group}
Every countable $(\mathscr{R},\lambda^{\mathbb{Z}})$-interpretable group is virtually abelian.
\end{prop}

\noindent Proposition~\ref{prop:lam-group} follows directly from several previous results.
 Every d-minimal expansion of $\bar{\mathbb{R}}$ admits definable selection by Miller~\cite{Miller-selection}.
Therefore an $(\mathscr{R},\lambda^{\mathbb{Z}})$-interpretable group is isomorphic to an $(\mathscr{R},\lambda^{\mathbb{Z}})$-definable group.
We now recall two results of Tychonievich.
The first is a special case of \cite[4.1.10]{Tychon-thesis}.

\begin{fact}\label{prop:ty}
If $X \subseteq \mathbb{R}^k$ is $(\mathscr{R},\lambda^{\mathbb{Z}})$-definable and countable, then there is an $\bar{\mathbb{R}}_\lambda$-definable surjection $f : (\lambda^\mathbb{Z})^m \to X$ for some $m$.
\end{fact}

\noindent Fact~\ref{prop:induced} is a minor rewording of \cite[4.1.2]{Tychon-thesis}.

\begin{fact}\label{prop:induced}
Every $(\mathscr{R},\lambda^{\mathbb{Z}})$-definable subset of $(\lambda^{\mathbb{Z}})^m$ is $(\lambda^{\mathbb{Z}},<,\cdot)$-definable.
\end{fact}

\noindent Facts \ref{prop:ty} and \ref{prop:induced} together imply that every countable $(\mathscr{R},\lambda^{\mathbb{Z}})$-definable group is isomorphic to a $(\mathbb{Z},<,+)$-definable group.
Now apply the following result of Onshuus and Vicaria~\cite{Onshuus-pres} to complete the proof of Proposition~\ref{prop:lam-group}.

\begin{fact}
Every $(\mathbb{Z},<,+)$-definable group is virtually abelian.
\end{fact}

\bibliographystyle{abbrv}
\bibliography{HW-Bib}

\end{document}